\documentclass[11pt]{amsart}

\usepackage{eucal}
\usepackage{amssymb,url,amscd,amsfonts}
\usepackage{latexsym}
\usepackage{hyperref}
\usepackage[all]{xy}
\theoremstyle{plain}
\newtheorem{theorem}{Theorem}

\newtheorem{corollary}[theorem]{Corollary}
\newtheorem{lemma}[theorem]{Lemma}

\numberwithin{equation}{section} 
\numberwithin{theorem}{section}

\begin{document}

\title[Radial graphs of constant curvature]
{Radial graphs of constant curvature and prescribed boundary}

\thanks{Work partially supported by CAPES-Brazil}

\author[F. F. Cruz]{Fl\'avio F. Cruz}
\address{Departamento de Matem\'atica\\ Universidade Regional do Cariri\\ Campus Crajubar \\ 
Juazeiro do Norte, Cear\'a\\ Brazil\\ 63041-141}
\email{flavio.franca@urca.br}

\subjclass[2000]{53C42, 35J60}

\begin{abstract}
In this paper we are concerned with the problem of finding hypersurfaces 
of constant curvature and prescribed boundary in the Euclidean space, 
without assuming the convexity of the prescribed solution and
using the theory of fully nonlinear elliptic equations. 
If the given data admits a suitable radial graph as a subsolution, 
then we prove that there exists a radial graph with constant curvature and realizing the prescribed boundary. 
As an application, it is proved that if $\Omega\subset\mathbb{S}^n$ is a mean convex domain  
whose closure is contained 
in an open hemisphere of $\mathbb{S}^n$ then, for $0<R<n(n-1),$  there exists
a radial graph of constant scalar curvature $R$ and boundary $\partial\Omega.$

\end{abstract}

\maketitle


\section{Introduction}
\label{section1}

The aim of this work is to study the following Plateau type problem:
given a smooth symmetric function $f$ of $n$ ($n\geq 2$) variables and a 
$(n-1)$-dimensional compact embedded submanifold $\Lambda$ of 
$\mathbb{R}^{n+1},$ decide whether there exists a hypersurface $\Sigma$
of constant curvature
\begin{equation}
\label{intro1}
f(\kappa[\Sigma])=c
\end{equation}
with boundary 
\begin{equation}
\label{intro2}
\partial \Sigma = \Lambda,
\end{equation}
where $\kappa[\Sigma]=(\kappa_1, \ldots, \kappa_n)$ denotes the principal curvatures
of $\Sigma$ and $c$ is a constant. The classical Plateau problem for minimal or 
constant mean curvature surfaces, as well as the corresponding problem for Gauss or scalar curvature,
are important particular cases of the problem.

Although the solvability of the problem in this generality still remains open, there are
various existence results for some important particular cases. These results have shown that
the theory of nonlinear elliptic PDEs is a powerful tool in order to understand the solvability of the problem.  
In order to apply the PDE techniques a successful strategy is describe the hypersurface $\Sigma$ 
as the graph of a solution of the Dirichlet problem associated to a certain PDE. 
After the works of Bernstein, Leray, Jenkins, Finn and others, Serrin applied this approach in \cite{SERRIN} and
proved the existence of hypersurfaces of constant mean curvature 
and prescribed boundary $\Lambda$ in two geometric settings: Firstly, 
when the boundary $\Lambda$ is a (vertical) graph over the boundary of a domain in a hyperplane and, 
secondly, when $\Lambda$ is a radial graph over the boundary of a domain in a hypersphere.
For more general curvature functions, the first breakthroughs about the solvability 
of the problem were due to Caffarelli, Nirenber and Spruck \cite{CNSV}. Applying the techniques developed in \cite{CNSI}
and \cite{CNSIII}, they proved the existence of solutions to \eqref{intro1}-\eqref{intro2}
for a large class of curvature functions, which includes the scalar and Gaussian curvature.
Importantly, however, they only treat the cases where the boundary date is constant,
$\Lambda$ is the boundary of a strictly convex domain in a hyperplane and the solution is a graph over this domain.
,For the specific case of high order mean curvature functions, Ivochkina \cite{IVO2} was able
to extend the existence for general boundary values and nonconvex domains. 
Much subsequent work aimed to improve and extend they results, as we can see in \cite{GUAN-LI},
\cite{IVO1}, \cite{JORGE-FLAVIO}, \cite{SHENG-URBAS-WANG}
and \cite{TRU1}. Later, Guan and Spruck \cite{GUAN-SPRUCK-1} established
existence results for constant Gaussian curvature hypersurfaces which are radial graphs over
a domain in a hypersphere and whose boundary is a radial graph over the boundary of the domain. 
Their results were extended in \cite{GUAN-SPRUCK-2} and \cite{TRU-WANG} to
convex curvature functions. However, the existing results to date leave open the 
case of radial graphs with constant non-convex curvature
functions. In particular, there is no result for the fundamental case of the scalar curvature in this context.
The main purpose of this paper is to establish the existence of radial graphs with constant higher order curvature
$f=H_r$, when the prescribed hypersurface is not assumed convex.
 In particular, our results embrace the scalar curvature case.

Let us now explain more precisely the framework we are considering. 
Let $\Omega$ be a smooth domain in $\mathbb{S}^n\subset \mathbb{R}^{n+1}$ with boundary 
$\partial \Omega.$
In order to solve the problem \eqref{intro1}-\eqref{intro2} we seek for a smooth hypersurface 
$\Sigma$ that can be represented as a radial graph
\begin{equation}
X(x)=\rho(x) x, \quad \rho>0,\,  x\in \bar{\Omega} 
\end{equation}
with prescribed curvature and boundary 
\begin{align}
\label{equation}
\begin{split}
f(\kappa_{\Sigma}[X]) &=\psi(x), \quad x\in\Omega\\
X(x) & =\phi(x)x, \quad x\in\partial\Omega
\end{split}
\end{align}
where $\kappa_{\Sigma}[X]=(\kappa_1, \ldots, \kappa_n)$ denotes the principal curvatures of $\Sigma$
at $X(x)$ with respect to the inward unit normal, $\psi\in C^{\infty}(\overline\Omega), \phi\in C^{\infty}
(\partial\Omega),$  $\psi,\phi>0$ and $f$ is a high order curvature function
\begin{equation}
\label{Hr}
f(\kappa)=H_r(\kappa)=\frac{S_r(\kappa)}{S_r(1,\ldots, 1)}
\end{equation}
where $1< r\leq n$ and $S_r$ is the $r-$th order elementary symmetric function,
\begin{equation}
\label{Sr}
S_r(\kappa)=\sum \kappa_{i_1}\kappa_{i_2}\ldots \kappa_{i_r}
\end{equation}
the sum being taken over all increasing $k-$tuples $i_1, i_2, \ldots, i_k\subset\{ 1, \ldots, n\}.$

In this context, a function $\rho\in C^2(\overline\Omega)$ is 
called {\it admissible} if $\kappa_\Sigma[X]\in \Gamma_r$ at each point $X$ of its radial graph $\Sigma,$
where $\Gamma_r$ is the open convex cone in $\mathbb{R}^n$ with vertex at the origin and given by
\begin{equation}
\label{gamma-r}
\Gamma_r=\{\kappa\in \mathbb{R}^n\,:\, S_j(\kappa)>0, \, j=1, \ldots, r\}.
\end{equation}

We shall assume the existence of a suitable admissible subsolution:
there exists a smooth admissible radial graph 
$\bar\Sigma$: $\bar{X}(x)=\bar{\rho}(x)x$ 
over $\bar{\Omega}$ that is locally strictly convex (up to the boundary) in 
a neighbourhood of $\partial\Omega$ and satisfies
\begin{align}
\label{subsolution}
\begin{split}
f(\kappa_{\bar{\Sigma}}[\bar{X}])>&\psi(x)\quad\textrm{in } \Omega,\\
\bar{\rho}=&\phi\quad \textrm{on } \partial\Omega.
\end{split}
\end{align}

Our main result may be state as follows:

\begin{theorem}
\label{teorema1}
Let $\Omega$ be a smooth domain whose closure is contained in an open hemisphere
of $ \mathbb{S}^n.$ Suppose the mean curvature of 
$\partial\Omega$ as a submanifold of $\Omega,$ computed with respect to the 
unit normal pointing to the interior of $\Omega,$ is nonnegative. 
Then, under condition \eqref{subsolution}, there exists a smooth 
radial graph $\Sigma$ satisfying \eqref{equation}. 
\end{theorem}

In general, solutions to equation \eqref{equation} are not unique.
See for instance, Example 8.5.2 in \cite{LOPEZ}. 
It follows from the Gauss equation that the case of scalar curvature $R$ of $\Sigma$ is given by
$R=n(n-1) H_2$, therefore the scalar curvature case is
included in Theorem \ref{teorema1}.  Moreover, using the function
$\bar\rho=1$ as a subsolution, we obtain the following result:

\begin{corollary}
\label{corolario}
Let $\Omega\subset \mathbb{S}^n$ be as in Theorem \ref{teorema1}. 
Then, for $0<R<n(n-1),$ there exists
a radial graph $\Sigma$ of constant scalar curvature $R$ and boundary $\partial\Sigma=\partial\Omega.$
\end{corollary}

A central issue in solving \eqref{equation} is to derive {\it a priori}
$C^2$ estimates for admissible solutions.
The height and boundary gradient bounds follows from the existence of a subsolution and 
the assumption on the geometry of $\Omega.$ Hessian and gradient
interior estimates are obtained applying the results
of \cite{CNSIV} to a suitable auxiliary equation. 
Our main contribution here is the establishment of the second derivatives
estimates on the boundary without imposing any condition on the geometry of
$\Omega.$ As this estimate is of independent we describe it separately:

\begin{theorem}
\label{teorema1-2}
Let $\rho\in C^3(\Omega)\cap C^2(\bar\Omega)$ be an admissible solution of \eqref{equation}.
Suppose that there exists a smooth admissible 
subsolution $\bar{\rho}$ of \eqref{equation}, i.e., 
the radial graph $\bar\Sigma$: $\bar{X}(x)=\bar{\rho}(x)x$ satisfy
\begin{align}
\begin{split}
f(\kappa_{\bar{\Sigma}}[\bar{X}])>&\psi(x)\quad\textrm{in } \Omega\\
\bar{\rho}=&\phi\quad \textrm{on } \partial\Omega,
\end{split}
\end{align}
and $\bar\Sigma$ is locally strictly convex (up to the boundary) in 
a neighbourhood of $\partial\Omega.$ Then there exists a constant $C$
depending on $\sup_{\Omega}\bar\rho, 
\|\bar \rho \|_{C^2(\bar\Omega)}, $ the convexity of $\bar \Sigma$
in a neighbourhood of $\partial\Omega$ and other known data, that satisfies
\begin{equation}
|\nabla^2 \rho|< C \quad \textrm{on } \partial\Omega
\end{equation}
where $\nabla^2 \rho$ denotes the Hessian of $\rho.$
\end{theorem}

An outline of the paper is as follows.  In Section \ref{section2} we list some basic formulae
which are needed later and define two elliptic operators to express \eqref{equation}. 
In Section \ref{section3} we deal with the {\it a priori} estimates 
for prospective solutions and prove Theorem \ref{teorema1-2}.
Finally in Section \ref{section4} we complete the proof of Theorem \ref{teorema1}
using the continuity method and a degree theory argument with the aid of the established  estimates.


\section{Preliminaries}
\label{section2}

In this section we recall the expressions of the second fundamental  form and other relevant geometric
quantities of a smooth radial graph $\Sigma$ given by $X(x)=\rho(x)x,$ where $\rho$ is a smooth
function defined in a domain $\Omega$ of the unit sphere 
$\mathbb{S}^n\subset\mathbb{R}^{n+1}.$

Let $e_1, \ldots, e_n$ be a smooth local orthonormal frame field on $\mathbb{S}^n$ and let $\nabla$ denote
the covariant differentiation on $\mathbb{S}^n.$ 
The metric of $\Sigma$ is then given in terms of $\rho$ by
\begin{equation}
g_{ij}=\langle\nabla_i X,\nabla_j X\rangle= \rho^2\delta_{ij}+\nabla_i\rho\nabla_j\rho, 
\end{equation}
where $\nabla_i=\nabla_{e_{i}}$ and $\langle \cdot , \cdot\rangle$ denotes the standard
 inner product in $\mathbb{R}^{n+1}.$ The interior unit normal to $\Sigma$ is
\begin{equation}
N=\frac{1}{(\rho^2+|\nabla \rho|^2)^{1/2}}(\nabla \rho-\rho x),
\end{equation}
where $\nabla\rho=\textrm{grad} \rho,$ and the second fundamental form of $\Sigma$ is
\begin{equation}
h_{ij}=\langle \nabla_{ij} X, N\rangle = \frac{1}{(\rho^2+|\nabla \rho|^2)^{1/2}}
(\rho^2\delta_{ij} +2\nabla_i\rho\nabla_j\rho-\rho\nabla_{ij}\rho),
\end{equation}
where $\nabla_{ij}=\nabla_i\nabla_j.$ 

Setting $u=1/\rho$ we can rewrite the expressions of the metric, its inverse 
and second fundamental form of $\Sigma$ at $X(x)= \frac{1}{u(x)}x$ in terms of $u$ by
\begin{equation}
g_{ij}=\frac{1}{u^2}(\delta_{ij}+\frac{1}{u^2}\nabla_iu\nabla_j u),
\quad g^{ij}=u^2\left(\delta_{ij}-\frac{\nabla_i u\nabla_j u}{w^2}\right)
\end{equation}
and
\begin{equation}
\label{segunda-forma}
h_{ij}=\frac{1}{uw}(u\delta_{ij}+\nabla_{ij}u),
\end{equation}
respectively, where $w=\sqrt{u^2+|\nabla u|^2}.$ The principal curvatures of $\Sigma$ are the
eigenvalues of the Weingarten matrix  $[h_i^j]=[g^{jk}h_{ki}].$ However, as in \cite{CNSIV}, 
here we will work with its similar symmetric matrix $A[u]=[a_{ij}]=[\gamma^{ik}h_{kl}\gamma^{jl}],$ 
where $[\gamma^{ij}]$ and its inverse matrix $[\gamma_{ij}]$ are given, respectively, by
\begin{equation}
\label{raiz-gij}
\gamma^{ij}= u\delta_{ij}-u\frac{\nabla_i u\nabla_j u}{w(u+w)}
\end{equation}
and
\begin{equation}
\label{raiz-gij2}
\gamma_{ij}= \frac{1}{u}\delta_{ij}+\frac{\nabla_i u\nabla_j u}{u^2(u+w)}.
\end{equation}
Geometrically, $[\gamma_{ij}]$ is the square root of the metric, i.e., $\gamma_{ik}\gamma_{kj}=g_{ij}.$ 

Now we present a reformulation of equation \eqref{equation} in the form
\begin{equation}
\label{equation2}
G(\nabla^2 u,\nabla u, u)=\tilde{\psi}(x),
\end{equation}
where $\tilde{\psi}=\psi^{1/r}.$
Let $\mathcal{S}$ be the space of all symmetric matrices and 
$\mathcal{S}_{r}$ the open subset of those symmetric matrices
$A\in\mathcal S$  for which the eigenvalues are
contained in $\Gamma_r.$  We define the function $F$ by
\begin{equation}
F(A)=f\big(\lambda(A)\big)=H_r^{1/r}\big(\lambda(A)\big), \quad A\in\mathcal{S}_{r}
\end{equation}
where $\lambda(A)=(\lambda_1, \cdots, \lambda_n)$ are the eigenvalues of $A.$ 
In the sequel we use $f$ to denote both $H_r$ and $H_r^{1/r}$.
Thus equation \eqref{equation} can be written in the form
\begin{equation}
F(A[u])=\tilde{\psi}\big(X(x)\big).
\end{equation}
Therefore, the function $G$ in \eqref{equation2} is defined by
\begin{equation}
G(\nabla^2 u,\nabla u, u)=F(A[u])
\end{equation}
and equation \eqref{equation} can be rewritten as
\begin{align}
\begin{split}
\label{equation3}
G(\nabla^2 u,\nabla u, u)&=\tilde{\psi} (x) \quad\textrm{in}\, \Omega,\\
u&=\varphi \quad\textrm{on}\, \partial\Omega,
\end{split}
\end{align}
where $\varphi=1/\phi.$

Next we will describe some properties of the functions $F$ and $G.$ 
We denote the first derivatives of $F$ by
\begin{equation*}
F^{ij}(A)=\frac{\partial F}{\partial a_{ij}}(A).
\end{equation*}
Since $H_r$ is positively monotonous on $\Gamma_r ,$ the symmetric matrix $[F^{ij}(A)]$
is positive definite  for any $A\in\mathcal{S}_{r}$
and it follows from the concavity of $H_r^{1/r}$ that $F$ is a concave function in $\mathcal{S}_{r}.$
$[F^{ij}(A)]$ and $A$ can be orthogonally diagonalized simultaneously. Thereafter, we have
\begin{equation}
\label{euler}
F^{ij}(A)a_{ij}=\sum f_i\lambda_i\leq f(\lambda)
\end{equation}
where the last inequality follows from the concavity of $f.$ Also we point out that
\begin{align}
\label{new} \sum f_i(\lambda)\lambda_i^2\leq C_0 (\lambda_j \mathbf{1}_{\lambda_j>0}&+
\sum_{k\neq j} f_k(\lambda)\lambda_k^2), \quad\textrm{for all } \lambda\in \Gamma_\psi,
\end{align}
where $\Gamma_\psi=\{\lambda\in\Gamma : \psi_0\leq f(\lambda)\leq\psi_1\}$ and
$C_0$ is a positive constant depending on $\psi_0$ and $\psi_1.$ This inequality was first proved
by Ivochkina in \cite{IVO2}. Using the expression for $A[u]$ we compute
\begin{equation}
\label{G-I-J}
G^{ij}=\frac{\partial G}{\partial\nabla_{ij}u}=\frac{1}{uw}\sum_{k,l}F^{kl} \gamma^{ik}\gamma^{jl}.
\end{equation}
Then equation \eqref{equation3} is elliptic for $A[u]\in\mathcal{S}_{r}.$
The concavity of $F$ implies that
$G$ is concave with respect to $\nabla_{ij}u.$ By assumption \eqref{subsolution}, the function
$\underline{u}=1/\bar\rho$ is a subsolution of equation \eqref{equation3}, i.e.,
\begin{align}
\begin{split}
\label{subsolution2}
G(\nabla^2 \underline  u,\nabla \underline u, \underline  u)&=\underline \psi(x)>\tilde{\psi}(x)
\quad\textrm{in}\, \Omega,\\
\underline  u&=\varphi \quad\textrm{on}\, \partial\Omega.
\end{split}
\end{align}

In order to establish the existence of solution for \eqref{equation} we will apply
the continuity method and a degree theory argument on two auxiliary forms of \eqref{equation3}.
Consider, for each fixed $t\in [ 0,1],$ the functions $\Psi^t$ and $\Xi^t$ defined in 
$\Delta= \{ X\in \mathbb{R}^{n+1} \, : \, \frac{X}{\|X\|}\in\overline\Omega\}$ by
\begin{equation}
\label{psi-t}
\Psi^t(\rho x)= \left(\frac{\overline \rho(x)}{\rho}\right)^3 \big(t\psi(x)+(1-t)
\underline{\psi}(x)\big)
\end{equation}
and
\begin{equation}
\label{phi-t}
\Xi^t(\rho x)= t\psi(x)+(1-t)\left(\frac{\overline \rho(x)}{\rho}\right)^3\psi(x).
\end{equation}
We shall work on the two corresponding auxiliary forms of \eqref{equation3}. In sections \ref{section3} 
and \ref{section4} we will represent generically  these equations by
\begin{align}
\begin{split}
\label{equation4}
G(\nabla^2 u,\nabla u, u)&=\Upsilon (X(x)) \quad\textrm{in}\, \Omega,\\
u&=\varphi \quad\textrm{on}\, \partial\Omega,
\end{split}
\end{align}
where $X(x)= \frac{1}{u(x)} x$ and $\Upsilon$ denotes a general positive smooth function on
$\Delta$. We finalize this section observing that the concavity of $f$ implies
\begin{equation}
\label{H-positiva}
\sum \lambda_i >0,
\end{equation}
for any $\lambda=(\lambda_1, \ldots, \lambda_n)\in\Gamma_{\psi}$ (e.g., \cite{CNSIII}).


\section{A priori estimates}
\label{section3}

In this section we obtain the {\it a priori} $C^2$ estimates for admissible solutions $u$ 
of \eqref{equation4} satisfying $u\geq \underline{u}.$ 

In order to derive an upper bound for $u,$ we note
that as the closure of $\Omega$ is contained in an open hemisphere and 
the mean curvature of $\partial\Omega$ is nonnegative, 
there exist \cite{SERRIN} 
a minimal radial graph $\hat{\Sigma}: \hat{X}(x)=\underline{\rho}(x)x$ over $\Omega$ with
boundary value $\underline{\rho}=\varphi.$ On the other hand, as \eqref{H-positiva} implies that 
the mean curvature of $\Sigma$ is positive,
we can apply the comparison principle to obtain $u\leq \overline{u},$ where 
$\overline{u}=1/\underline{\rho}.$
Then $\underline{u}\leq u\leq \overline{u}$ in $\Omega$ and $\underline{u}=u=\overline{u}$ on
$\partial\Omega,$ which yields the height and the boundary gradient bounds.
For the interior gradient estimate we first observe that
\begin{equation}
\label{derivada-ro}
\frac{\partial}{\partial\rho}(\rho\Upsilon(\rho x))\leq 0\quad \textrm{if} \quad \rho\leq \bar\rho.
\end{equation}
Therefore, the interior gradient bounds
can be established as in \cite{CNSIV} (see also \cite{ABL}). Then we
have the following result.
\begin{theorem}
\label{teorema2}
Let $u\geq \underline{u}$ an admissible solution of \eqref{equation4}. Then we have the estimates
\begin{equation}
\label{S3-22}
L^{-1}\leq u\leq L, \quad |\nabla u|\leq C \quad \textrm{in}\,\, \bar\Omega,
\end{equation}
where $L$ and $C$ depend on $\inf_{\Omega}\underline u, \|\underline u\|_{C^1(\bar\Omega)}$
and other known data.
\end{theorem}

The only places we need assumptions on the geometry of $\Omega$
is in getting an upper bound for $u$ and 
the gradient boundary estimate. In what follows, when we use $L,$ it always means 
the same constant and we will denote
\begin{equation*}
\Delta_L=\{X\in \Delta
\, : \, L^{-1}\leq \|X\|\leq L\}.
\end{equation*}

Now we shall to establish the second derivatives estimates. First we will obtain
bounds for $|\nabla^2 u|$ on $\partial \Omega.$ 

Consider an arbitrary point $x_0\in\partial\Omega,$ let $e_1, \ldots, e_n$ be a local orthonormal frame 
field on $\mathbb{S}^n$ around $x_0,$ obtained by parallel translation of a local orthonormal frame field on
$\partial\Omega$ and the interior, unit, normal vector field to $\partial\Omega,$ along the geodesic
perpendicular to $\partial\Omega$ on $\mathbb{S}^n.$ We assume that $e_n$ is the parallel translation of 
the unit normal vector field on $\partial\Omega.$

As $u=\varphi$ on $\partial\Omega$ we have
\begin{equation}
\label{U-2-bordo}
\nabla_{ij}(u-\varphi)=-\nabla_{n}(u-\varphi)
B_{ij}\quad\textrm{for} \, i,j<n,
\end{equation} 
where $B_{ij}=\langle \nabla_{e_{i}} e_j, e_n\rangle$ is the second fundamental form of
$\partial\Omega.$ It follows that
\begin{equation}
\label{EST-TANG}
|\nabla_{ij} u(x_0)|\leq C, \quad i,j<n,
\end{equation}
for a uniform constant $C.$ 

We now proceed to estimate the mixed tangential-normal derivatives 
$\nabla_{k n} u(x_0), \, k<n.$ 
By a straightforward computation, we get
\begin{align}
\label{Gs}
\begin{split}
G^{s}&=\frac{\partial G}{\partial \nabla_s u}=-\frac{1}{w^2}F^{ij}a_{ij}\nabla_s u
-\frac{2}{u^2w} F^{ij}a_{ik}\gamma^{ks}\nabla_j u
\end{split}
\end{align}
and
\begin{align}
\label{Gu}
\begin{split}
G_u=&\frac{\partial G}{\partial u}= \frac{|\nabla u|^2}{uw^2}F^{ij}a_{ij}+\frac{2}{uw^2}\sum_kF^{ij}a_{ik}\nabla_j u
\nabla_ku \\ &+\frac{u}{w}F^{ij}\Big(\delta_{ij}-\frac{\nabla_i u\nabla_j u}{w^2}\Big).
\end{split}
\end{align}
In particular, as $[F^{ij}]$ and $[a_{ij}]$ can be diagonalized simultaneously, it follows
from Theorem \ref{teorema2} that
\begin{align}
\begin{split}
\label{bound-gs}
|G^s|\leq & C\big(1+\sum f_i |\kappa_i|\big) \\
|G_u|\leq C&\big(1+\sum (f_i |\kappa_i| + f_i)\big),
\end{split}
\end{align}
for a uniform constant $C$ depending on $\|u\|_{C^{1}(\bar\Omega)}$ and $\sup_{\Delta_L}\psi^t.$

Now we present some key preliminary lemmas.  
Let $\varrho(x)$ denote the distance from $x\in\Omega$ to $x_0,$ 
$\varrho(x)=\textrm{dist}_{\mathbb{S}^n}(x,x_0),$ and set
\begin{equation*}
\Omega_{\delta}=\{x\in\Omega\, :\, \varrho(x)<\delta\}.
\end{equation*}
Since $\nabla_{ij}\varrho^2(x_0)= 2\delta_{ij}$, by choosing $\delta>0$ sufficiently small we 
can assume that $\varrho$ is smooth in $\Omega_{\delta},$
\begin{equation}
\label{RHO-2}
\delta_{ij}\leq \nabla_{ij}\varrho^2\leq 3\delta_{ij} \quad\textrm{in}\,\, \Omega_\delta,
\end{equation}
and the distance function $d(x)=\textrm{dist}_{\mathbb{S}^n}(x,\partial\Omega)$ to the boundary 
$\partial\Omega$ is smooth in $\Omega_\delta.$ 

\begin{lemma}
\label{DEF-W}
For some positive constants $K$ and $M$  sufficiently large depending on $
\|u\|_{C^{1}(\bar\Omega)}, \|\Upsilon\|_{C^1(\Delta_L)}$ and other known data,
the function
\begin{equation}
\label{CARA-W}
\Phi=\nabla_k (u-\varphi) -\frac{K}{2}\sum_{l<n}\big(\nabla_l (u-\varphi)\big)^2
\end{equation}
satisfies
\begin{equation}
\label{LW}
G^{ij}\nabla_{ij} \Phi\leq M(1+|\nabla \Phi|+G^{ij}\delta_{ij}+G^{ij}\nabla_i \Phi\nabla_j \Phi) 
\quad \textrm{in}  \quad \Omega_\delta.
\end{equation}
\end{lemma}

\begin{proof} 
A straightforward computation yields
\begin{align*}
G^{ij}&\nabla_{ij} \Phi =  G^{ij} \nabla_{ijk}u
-K \sum_{l<n}\nabla_l (u-\varphi)G^{ij}\nabla_{ijl}u-G^{ij}\nabla_{ijk}\varphi 
\\&-K \sum_{l<n}G^{ij}\nabla_{li} (u-\varphi)\nabla_{lj} (u-\varphi)
+K \sum_{l<n}\nabla_{l}(u-\varphi)G^{ij}\nabla_{ijl}\varphi.
\end{align*}
Then, using Theorem \ref{teorema2} we easily get the bound
\begin{align}
\label{conta-v3-2}
\begin{split}
G^{ij}&\nabla_{ij} \Phi \leq  G^{ij} \nabla_{ijk}u
-K \sum_{l<n}\nabla_l (u-\varphi)G^{ij}\nabla_{ijl}u
\\ & -K \sum_{l<n}G^{ij}\nabla_{li} (u-\varphi)\nabla_{lj} (u-\varphi)+ C\sum G^{ii},
\end{split}
\end{align}
for a uniform constant $C$ depending on $\|u\|_{C^{1}(\bar\Omega)}, 
\|\varphi\|_{C^{3}(\partial\Omega)}$ and $K.$
Differentiating equation \eqref{equation4} we get
\begin{align}
\label{CONTA-PSI-K-1}
G^{ij}\nabla_{pij} u+G^s\nabla_{ps}u+G_u\nabla_p u=\nabla_p\Upsilon.
\end{align}
Hence, applying the standard formula for commuting the order of covariant derivatives 
on $\mathbb{S}^n$ we obtain
\begin{align}
\label{v3-psik-2}
\begin{split}
G^{ij}\nabla_{ijp} u &= G^{ij}(\nabla_{pij}u+\delta_{ij}\nabla_pu-\delta_{pj}\nabla_i u)\\
 &= -G^s\nabla_{sp}u+G_u\nabla_p u
+G^{ij}(\delta_{ij}\nabla_pu-\delta_{pj}\nabla_i u)+\nabla_p\Upsilon.
\end{split}
\end{align}
Thus, as
\begin{align*}
\begin{split}
G^{s}\nabla_{sk} u - &K \sum_{l<n}\nabla_l (u-\varphi)G^s\nabla_{sl}u=
G^s\nabla_s \Phi+G^s\nabla_{ks}\varphi \\ &-K \sum_{l<n}\nabla_l (u-\varphi)G^s\nabla_{sl}\varphi,
\end{split}
\end{align*}
we get
\begin{align*}
\begin{split}
&G^{ij}\nabla_{ijk}u-K \sum_{l<n}\nabla_l (u-\varphi)G^{ij}\nabla_{ijl}u =  -G^s\nabla_s \Phi
-G^s\nabla_{ks}\varphi \\ &+K \sum_{l<n}\nabla_l (u-\varphi)\big(G^s\nabla_{sl}\varphi
-G_u\nabla_l u-G^{ij}(\delta_{ij}\nabla_l u-\delta_{lj}\nabla_i u )-\nabla_l\Upsilon\big)
\\ &+G_u\nabla_ku+G^{ij}(\delta_{ij}\nabla_k u-\delta_{kj}\nabla_i u)+\nabla_k\Upsilon.
\end{split}
\end{align*}
Therefore, replacing this expression into \eqref{conta-v3-2} and using \eqref{S3-22}  and \eqref{bound-gs}
 we find
\begin{align}
\label{conta-v3-4}
\begin{split}
G^{ij}&\nabla_{ij} \Phi \leq   -G^s\nabla_s \Phi -K \sum_{l<n}G^{ij}\nabla_{li} (u-\varphi)\nabla_{lj} (u-\varphi)
\\& + C\big(1+ \sum (f_i |\kappa_i| + f_i)\big).
\end{split}
\end{align}

Let $P=[\eta_{ij}]$ be an orthogonal matrix that simultaneously diagonalizes $[F^{ij}]$ and
$[a_{ij}],$ and let $\{\tau_1, \ldots, \tau_n\}$ be a basis of vectors that induce by the
parametrization  $X$ a basis of principal vectors of $\Sigma,$ that is,
a basis of eigenvectors of the Weingarten operator of $\Sigma.$
Henceforth, we will use the greek letters for derivatives in the basis $\tau_1, \ldots, \tau_n$ 
and latin letters for derivatives in the frame $e_1, \ldots, e_n.$ For instance,
$\nabla_{\alpha\beta}u$ and $\nabla_{s\alpha}u$ will denote respectively
$\nabla^2u(\tau_\alpha, \tau_\beta)$ and $\nabla^2 u(e_s, \tau_\alpha).$ In particular, as
$\gamma_{\alpha\beta}$ is the unique positive square root of $g_{\alpha\beta},$ we have
\[
g_{\alpha\beta}=\langle \nabla_\alpha X,\nabla_\beta X\rangle=\gamma_{\alpha\beta}=\delta_{\alpha\beta}.
\]
Thus, inequality \eqref{LW} can be written as
\begin{align}
\begin{split}
\label{LW-2}
\sum f_\alpha \nabla_{\alpha\alpha}\Phi \leq  M\left(1+|\nabla \Phi|+
\sum (f_\alpha(\nabla_\alpha \Phi)^2+ f_\alpha)\right).
\end{split}
\end{align}
In the sequel, we will often denote by $C$ a uniform constant under control. 
As $\nabla_{ij} u=uw\sum_{k,l}\gamma_{ik}a_{kl}\gamma_{jl}-u\delta_{ij},$ it follows from 
\eqref{S3-22}  and \eqref{bound-gs} that
\begin{align}
\label{S2-1}
G^{ij}\nabla_{li} (u-\varphi)\nabla_{lj} (u-\varphi)\geq \sum_\alpha \big( \theta_{0} f_\alpha
\kappa_\alpha^2\eta_{l\alpha}^2-C (f_\alpha|\kappa_\alpha|+  f_\alpha)\big),
\end{align}
for a positive uniform constant $\theta_0.$ Similarly,  applying Theorem \ref{teorema2} and 
inequality \eqref{Gs} we get the bound
\begin{align}
\label{S2-2}
 |G^s\nabla_s \Phi|\leq C \big( |\nabla\Phi|+\sum f_\alpha|\kappa_\alpha \nabla_\alpha \Phi|\big).
\end{align}
Then we replace \eqref{S2-1}-\eqref{S2-2} into \eqref{conta-v3-4} to obtain
\begin{align}
\label{S2-3}
\begin{split}
\sum f_\alpha \nabla_{\alpha\alpha} \Phi &\leq\sum\big( Cf_\alpha|\kappa_\alpha \nabla_\alpha \Phi|
-K\theta_0\sum_{l<n} f_\alpha\kappa_\alpha^2\eta_{l\alpha}^2\big)
\\& + C\big(1+  |\nabla \Phi|+\sum (f_\alpha |\kappa_\alpha|+ f_\alpha)\big).
\end{split}
\end{align}

Now let us consider two cases. First we assume that, for all $\alpha\in\{1,\ldots, n\},$ it holds
\begin{equation}
\label{I}
\sum_{l<n}\eta_{l\alpha}^2\geq K^{-2}.
\end{equation}
The second case occurs when  (\ref{I}) does not hold.
In the first case
\begin{equation}
\label{I1}
-\sum_\alpha\sum_{l<n} f_\alpha\kappa_\alpha^2 \eta_{l\alpha}^2
\leq -\sum_\alpha f_\alpha\kappa_\alpha^2.
\end{equation}
Then \eqref{LW-2} follows from \eqref{S2-3} and the inequalities
\begin{align}
\begin{split}
\label{I2}
\sum f_\alpha|\kappa_\alpha \nabla_\alpha \Phi | &\leq \sum \big( \epsilon f_\alpha\kappa_{\alpha}^2
+\epsilon^{-1} f_\alpha(\nabla_\alpha \Phi)^2\big)
\\ \sum f_\alpha |\kappa_\alpha| &\leq \sum \big(\epsilon   f_\alpha \kappa_{\alpha}^2+\epsilon^{-1} f_\alpha\big), 
\end{split}
\end{align}
for an appropriate constant $\epsilon>0.$

In the second case, there exists some $1\leq \gamma\leq n$ such that
\begin{equation}
\label{delta}
\sum_{l<n}\eta_{l\gamma}^2< K^{-2}.
\end{equation}
As in \cite{IVO1}, we can prove (see the appendix) that \eqref{delta} implies
\begin{equation}
\label{delta2}
\sum_{l<n}\eta_{l\alpha}^2\geq\epsilon_0
\end{equation}
for all $\alpha\neq \gamma,$ where $\epsilon_0>0$ is a uniform positive constant that does
not depends on $K.$ For simplicity, let us assume that $\gamma=1.$ 
To proceed we consider two subcases: $\kappa_1\leq 0$ 
and $\kappa_1>0$. \\
If $\kappa_1\leq 0$ then inequality \eqref{new} yields
\begin{equation*}
f_1\kappa_1^2\leq  C_0\sum_{\alpha>1}f_\alpha\kappa_\alpha^2.
\end{equation*}
Hence
\begin{equation}
\label{C3}
\sum  f_\alpha\kappa_\alpha^2 \leq C \sum_{\alpha>1} f_\alpha \kappa_\alpha^2
\end{equation}
and we can estimate
\begin{align}
\begin{split}
\label{II2}
\sum f_\alpha| \kappa_\alpha\nabla_\alpha \Phi | &\leq  \sum f_\alpha\big(\nabla_\alpha \Phi\big)^2+
C\sum_{\alpha>1} f_\alpha\kappa_{\alpha}^2 \\ \sum f_\alpha |\kappa_\alpha| 
&\leq \sum  f_\alpha+ C \sum_{\alpha> 1} f_\alpha \kappa_{\alpha}^2.
\end{split}
\end{align}
On the other hand, it follows from \eqref{delta2} that
\begin{equation}
\label{C4}
\sum_\alpha \sum_{l<n} f_\alpha\eta_{l\alpha}^2\kappa_\alpha^2 
\geq \epsilon_0\sum_{\alpha>1} f_\alpha\kappa_\alpha^2.
\end{equation}
Applying \eqref{II2} and \eqref{C4} 
into \eqref{S2-3} we then obtain \eqref{LW-2} by choosing $K$ sufficiently
 large. 

Now suppose that $\kappa_1>0.$ A straightforward computation and the expression
$\nabla_{\alpha\beta} u=u(w\kappa_\alpha\delta_{\alpha\beta}-\sigma_{\alpha\beta}),$ where
$\sigma_{\alpha\beta}=\langle \tau_\alpha, \tau_\beta\rangle, $ yield
\begin{align*}
\begin{split}
f_1|\kappa_1 \nabla_1\Phi| &=
 f_1\kappa_1\big|\sum_{\alpha}\eta_{k\alpha}\nabla_{1\alpha }u-\nabla_{1k}\varphi
 -K\sum_{\alpha}\sum_{l<n}\eta_{l\alpha}\nabla_l
 (u-\varphi)\nabla_{1\alpha}(u-\varphi)\big| \\ &
\leq  uw\big|\eta_{k1}-K\sum_{l<n}\eta_{l1}\nabla_l (u-\varphi)\big| f_1\kappa_1^2
+ Cf_1\kappa_1.
\end{split}
\end{align*}
Moreover,
\begin{align*}
\begin{split}
|\nabla\Phi| & \geq |\nabla_1\Phi|=\big|\sum_{\alpha}\eta_{k\alpha}\nabla_{1\alpha }u
-\nabla_{1k}\varphi-K\sum_{\alpha}\sum_{l<n}\eta_{l\alpha}\nabla_l
 (u-\varphi)\nabla_{1\alpha}(u-\varphi)\big| 
\\ & \geq \big|\sum_{\alpha}\eta_{k\alpha}u(w\kappa_\alpha -1)\delta_{\alpha 1}
-K\sum_{\alpha} \sum_{l<n}\eta_{l\alpha}\nabla_l (u-\varphi)u(w\kappa_\alpha -1)\delta_{\alpha 1}\big|-C\\
&\geq  uw\big|\eta_{k1}-K\sum_{l<n}\eta_{l1}\nabla_l (u-\varphi)\big|\kappa_1-C.
\end{split}
\end{align*}
Thus, applying \eqref{new} we get the bound
\begin{align}
\label{S2-6}
\begin{split}
f_1|\kappa_1 &\nabla_1\Phi| \leq  C(1+|\nabla \Phi|)+Cf_1\kappa_1 
\\ &+C\Big|\eta_{k1}-K\sum_{l<n}\eta_{l1}\nabla_l (u-\varphi)\Big| 
\sum_{\alpha>1}f_\alpha\kappa_\alpha^2.
\end{split}
\end{align}
On the other hand, inequality \eqref{delta} gives
\begin{equation*}
\Big|\eta_{k1}-K\sum_{l<n}\eta_{l1}\nabla_l (u-\varphi)\Big| \leq C_1,
\end{equation*}
for a uniform positive constant $C_1$ that does not depend on $K.$ Therefore
\begin{equation}
\label{S2-7}
f_1|\kappa_1 \nabla_1\Phi|\leq  C(1+|\nabla \Phi|)+Cf_1\kappa_1
+C_1\sum_{\alpha>1} f_\alpha \kappa_\alpha^2.
\end{equation}
To control the term $f_1\kappa_1$ we use \eqref{euler} to obtain
\begin{align}
\label{S2-8}
\begin{split}
f_1\kappa_1&= f_\alpha\kappa_\alpha -\sum_{\alpha>1}f_\alpha\kappa_\alpha
 \\ &\leq C\big(1+\sum f_\alpha\big)+\sum_{\alpha>1}f_\alpha\kappa_\alpha^2.
\end{split}
\end{align}
Then
\begin{equation*}
f_1|\kappa_1 \nabla_1\Phi|\leq  C(1+|\nabla \Phi|+\sum f_\alpha)
+C_1\sum_{\alpha>1} f_\alpha \kappa_\alpha^2
\end{equation*}
and we get the bound
\begin{align}
\label{S2-9}
\begin{split}
\sum f_\alpha|\kappa_\alpha \nabla_\alpha \Phi | & \leq C(1+|\nabla \Phi|+\sum f_\alpha)
+C_1\sum_{\alpha>1} f_\alpha \kappa_\alpha^2\\ &+
\sum_{\alpha>1} f_\alpha\kappa_{\alpha}^2+
\sum f_\alpha(\nabla_\alpha \Phi)^2.
\end{split}
\end{align}
Finally, as $\kappa_1>0$ we can use \eqref{S2-8} to get
\begin{align}
\label{S2-10}
\begin{split}
f_\alpha |\kappa_\alpha|=\big(f_1\kappa_1+\sum_{\alpha>1}f_\alpha |\kappa_\alpha |\big)\leq 
C\big(1+\sum f_\alpha\big)+\sum_{\alpha>1}f_\alpha \kappa_\alpha^2.
\end{split}
\end{align}
Hence, using \eqref{C4} and replacing \eqref{S2-9} and \eqref{S2-10} into \eqref{S2-3}, 
we obatin \eqref{LW-2} by choosing $K$ sufficiently  large. 
\end{proof}

Setting
\begin{equation}
\label{w-til}
\tilde{\Phi}= 1-e^{-a_0 \Phi}
\end{equation}
for a positive constant  $a_0$ large such that $a_0\geq M,$ where $M$ is the constant given
in \eqref{LW}, we get
\begin{align}
\label{LW-TIL}
G^{ij}\nabla_{ij}\tilde{\Phi} \leq   M(1+|\nabla \tilde{\Phi}|+G^{ij}\delta_{ij}).
\end{align}

Now we present the following improved version of Lemma 3.3 in \cite{SU}. In what follows, we denote by
$d(x)=\textrm{dist}(x,\partial\Omega)$ the distance function to the boundary.
\begin{lemma}
\label{DEF-V}
There exist some uniform positive constants $t, \delta , \varepsilon$ sufficiently small and 
$N$ sufficiently large depending on $\inf_{\bar\Omega}\underline u, 
\|\underline u\|_{C^2(\bar\Omega)}, \sup_{\Delta_L}\Upsilon,$ the convexity of $\bar\Sigma$
in a neighbourhood of $\partial\Omega$ and other known data, such that the function
\begin{equation}
\label{V}
\Theta=u-\underline{u} +td-Nd^2
\end{equation}
satisfies
\begin{equation}
\label{LV<0}
G^{ij}\nabla_{ij}\Theta\leq -(1+|\nabla \Theta|+ G^{ij}\delta_{ij}) \quad \textrm{in }  \Omega_\delta
\end{equation}
and
\begin{equation}
\label{theta-bordo}
\Theta\geq 0 \quad\textrm{on } \partial\Omega_\delta.
\end{equation}
\end{lemma}
\begin{proof}
As the surface $\bar{X}(x)=\frac{1}{\underline u}x$ is convex in a neighbourhood of $\partial\Omega,$ 
we can find $\beta>0$ and $\delta>0$ such that
\begin{equation}
\label{convex1}
[\underline{u} I+\nabla^2\underline{u}]\geq 4\beta I \quad \textrm{in } \,\Omega_{\delta}.
\end{equation}
In particular, $\lambda(\underline{u}I+\nabla^2\underline{u}-3\beta I)$
lies in a compact set of $\Gamma_n^+\subset\Gamma_r.$
Since $|\nabla d|=1$ and $-CI\leq \nabla^2d\leq C I,$ 
for a constant $C$ depending only on the geometry of $\Omega,$ we have 
\begin{equation}
\label{d1}
G^{ij}\nabla_{ij} d\leq C G^{ij}\delta_{ij}
\end{equation}
 and
\begin{align}
\begin{split}
\label{convex3}
\lambda(\underline{u}I+\nabla^2\underline{u}+N \nabla^2d^2-2\beta I) 
\geq \lambda(\underline{u}I+\nabla^2\underline{u}+2N \nabla d\otimes\nabla d-3\beta I)
\end{split}
\end{align}
in $\Omega_\delta,$ when $\delta$ is sufficiently small (so that $2N\delta<\beta/C).$ 
Using the concavity of $f$ we get
\begin{align*}
 F\Big( & \big[\frac{1}{uw}\gamma^{ik}(\underline{u}\delta_{kl}
+\nabla_{kl}\underline{u}+2N\nabla_l d\nabla_k d
-3\beta \delta_{lk})\gamma^{jl}\big]\Big)-\psi\big(X\big)\\
& \leq G^{ij}\Big(\nabla_{ij}\underline{u}+\underline{u}\delta_{ij} 
+N\nabla_{ij}d^2-2\beta\delta_{ij}-(u\delta_{ij}+\nabla_{ij}u)\Big) 
\\ & = G^{ij} \nabla_{ij}(\underline{u}-u+Nd^2)+(\underline{u}-u)G^{ij}\delta_{ij}
-2\beta G^{ij}\delta_{ij}.
\end{align*}
Then, using \eqref{d1}, \eqref{convex3} and that $u\geq \underline u,$  we get
\begin{align}
\begin{split}
\label{conta-convex1}
G^{ij}&\nabla_{ij}(u-\underline{u}+td-Nd^2)\leq \Upsilon\big(X\big)-2\beta G^{ij}\delta_{ij}
+tCG^{ij}\delta_{ij}\\ &- F\Big( \big[\frac{1}{uw}\gamma^{ik}
(\underline{u}\delta_{kl}+\nabla_{kl}\underline{u}+2N\nabla_l d\nabla_k d
-3\beta \delta_{lk})\gamma^{jl}\big]\Big) \\
 = & - f\Big( \lambda\big(\frac{1}{uw}\gamma^{ik}
(\underline{u}\delta_{kl}+\nabla_{kl}\underline{u}-3\beta \delta_{kl})\gamma^{jl}
+\frac{2N}{uw}\gamma^{ik}\nabla_l d\nabla_k d\gamma^{jl}\big)\Big)\\
&+(tC-2\beta)G^{ij}\delta_{ij}+\Upsilon\big(X\big) .
\end{split}
\end{align}
By the choice of $\beta$ and Theorem \ref{teorema2}, there exists a uniform positive constant 
$\lambda_0$ satisfying
\begin{equation}
\label{convex4}
\Big[\frac{1}{uw}\gamma^{jk}(\underline{u}\delta_{kl}
+\nabla_{kl}\underline{u}-3\beta \delta_{lk})\gamma^{jl}\Big]\geq \lambda_0 I.
\end{equation}
Then we can find a uniform positive constant $\mu_0$ such that
\begin{align*}
P^T[\frac{1}{uw}\gamma^{ik}
(\underline{u}\delta_{kl}+&\nabla_{kl}\underline{u}-3\beta \delta_{kl})\gamma^{jl}
+\frac{2N}{uw}\gamma^{ik}\nabla_l d\nabla_k d\gamma^{jl}\big]P \\ & \geq  
\textrm{diag}\{\lambda_0,\lambda_0, \ldots, \lambda_0+N\mu_0\},
\end{align*}
where $P$ is an orthogonal matrix that diagonalizes 
$\big[\frac{2N}{uw}\gamma^{kl}\nabla_l d\nabla_k d\gamma^{jl}\big].$
Then, by the ellipticity and concavity of $f$ we get
\begin{align*}
\begin{split}
f&\Big( \lambda\big(\frac{1}{uw}\gamma^{ik}
(\underline{u}\delta_{kl}+\nabla_{kl}\underline{u}-3\beta \delta_{lk})\gamma^{jl}
+\frac{2N}{uw}\gamma^{ik}\nabla_l d\nabla_k d\gamma^{jl}\big)\Big)\\
&=f\Big( \lambda\big(P^T\big[\frac{1}{uw}\gamma^{ik}(\underline{u}\delta_{kl}+
\nabla_{kl}\underline{u}-3\beta \delta_{lk})\gamma^{jl}
+\frac{2N}{uw}\gamma^{ik}\nabla_l d\nabla_k d\gamma^{jl}\big]P\big)\Big) \\ &
 \geq f(\lambda_0,\lambda_0, \ldots, \lambda_0+N\mu_0).
\end{split}
\end{align*} 
Since 
\begin{equation*}
f(\lambda_0,\lambda_0, \ldots, \lambda_0+N\mu_0)
\rightarrow +\infty \quad \textrm{as} \quad N\rightarrow+\infty,
\end{equation*}
it follows from \eqref{conta-convex1} 
that, for $t$ small enough such that 
$Ct\leq \beta$ and $N$ sufficient large, we have
\begin{equation*}
G^{ij}\nabla_{ij}\Theta\leq -C-3t-\beta G^{ij}\delta_{ij}
\end{equation*}
where $C$ is a uniform constant that satisfies $|\nabla(u-\underline u)|\leq C.$ Finally, choosing 
$\delta$ even smaller, such that $\delta N<t,$ we get $|\nabla \Theta|\leq C+3t$ 
and $\Theta\geq 0$ on $\partial(\Omega\cap \Omega_\delta).$ 
\end{proof}

We are now in position to derive the mixed second derivatives boundary estimate.
Consider the functions
\begin{equation}
\label{cara-w-barra}
\bar{\Phi}=\tilde{\Phi}+b_0(u-\underline{u})
\end{equation}
and
\begin{equation}
\label{cara-theta}
\bar{\Theta}=-c_0\varrho^2-d_0 \Theta,
\end{equation}
where $b_0, c_0$ and $d_0$ are positive constants to be chosen. 
Following the reasoning in the proof of Lemma \ref{DEF-V}, we can easily prove that
\begin{align}
\label{V-1}
G^{ij}\nabla_{ij}(u-\underline{u})\leq C-\beta G^{ij}\delta_{ij}
\end{align}
in $\Omega_\delta,$ for sufficiently small $\delta,\beta>0.$
Then we conclude from \eqref{LW-TIL} that
\begin{equation*}
G^{ij}\nabla_{ij}\bar{\Phi} \leq M(1+b_0+|\nabla \tilde{\Phi}|)+b_0\big(C-\beta G^{ij}\delta_{ij}\big).
\end{equation*}
Hence, choosing $b_0$ sufficiently large we get 
\begin{equation}
\label{INEL-PHI}
G^{ij}\nabla_{ij}\bar{\Phi}\leq M_0(1+|\nabla \bar{\Phi}|)\quad \textrm{in } \Omega_\delta,
\end{equation}
for a uniform positive constant $M_0=M_0(a_0,b_0, M).$

On the other hand, using Lemma \ref{DEF-V} and inequality \eqref{RHO-2}
we can estimate
\begin{align*}
G^{ij}\nabla_{ij}\bar{\Theta} =& -c_0G^{ij}\nabla_{ij}\varrho^2-d_0G^{ij}\nabla_{ij} d \\
& \geq -3c_0 G^{ij}\delta_{ij}+d_0\big(1+|\nabla \Theta|+\beta G^{ij}\delta_{ij}\big).
\end{align*}
As $|\nabla \bar{\Theta}|\leq 2\delta c_0 + d_0|\nabla \Theta|,$ choosing $d_0>>c_0$
sufficiently large, we get
\begin{equation}
\label{INEL-THETA}
G^{ij}\nabla_{ij}\bar{\Theta}\geq M_0(1+|\nabla \bar{\Theta}|) \quad \textrm{in } \Omega_\delta.
\end{equation}

Now we compare $\bar{\Phi}$ and $\bar{\Theta}$ on $\partial\Omega_\delta.$ 
At this point we need to assume that the index $k$ fixed in \eqref{CARA-W} where
$\Phi$ is defined, is chosen so that 
$1\leq k\leq n-1.$ In particular, $e_k$ is tangent along $\partial\Omega$
and we have $\bar{\Phi}=0$ on $\partial\Omega_\delta\cap \partial\Omega.$ 
By \eqref{theta-bordo} we have $\bar{\Theta}\leq -c_0\varrho^2$ on 
$\partial\Omega_\delta,$ then
$\bar{\Phi}=0\geq  -c_0\varrho^2\geq \bar{\Theta}$ on $\partial\Omega_\delta\cap \partial\Omega.$ 
For $\partial\Omega_\delta\cap\Omega,$
notice that $|\bar{\Phi}|\leq C$ on $\partial\Omega_\delta\cap\Omega$ for a uniform constant $C.$
Hence,  choosing $c_0$ sufficiently large we get
\begin{equation*}
\bar{\Theta}=-c_0\varrho^2-d_0 \Theta\leq -c_0 \delta^2\leq \bar{\Phi}
 \quad\textrm{on} \quad \partial\Omega_\delta\cap \Omega.
\end{equation*}
Therefore
\begin{equation}
\label{PHI-THETA}
\bar{\Theta}\leq \bar{\Phi} \quad \textrm{on} \, \partial\Omega_\delta.
\end{equation}

Finally, it follows from \eqref{INEL-PHI}, \eqref{INEL-THETA} and the Comparison Principle (see e.g. \cite{GT}) 
that $\bar{\Phi}\geq \bar{\Theta}$
in $\Omega_\delta.$ As $\bar{\Phi}(x_0)=\bar{\Theta}(x_0)$ we get 
$\nabla_n \bar{\Theta}(x_0)\leq \nabla_n \bar{\Phi}(x_0),$ which give us
\begin{equation}
\label{mista-explicita}
\nabla_{kn}u(x_0) \geq \nabla_{kn}\varphi(x_0)
-\frac{d_0}{a_0}\big(\nabla_n(u-\underline u)(x_0)+t\big).
\end{equation}
for $1\leq k\leq n-1.$
Then, as $x_0\in\partial\Omega$ is arbitrary, we get
\begin{equation}
\label{TANGENTE-NORMAL}
|\nabla_{kn} u|<C \quad \textrm{on}\, \partial\Omega.
\end{equation}
The mixed second derivatives boundary estimate is established.

Now we consider the pure normal second derivative bound.
Since $\Sigma$ has positive mean curvature, we only need to derive an upper bound
\begin{equation}
\label{NORMAL-NORMAL}
\nabla_{nn}u < C \quad \textrm{on}\,\, \partial\Omega.
\end{equation}

Let $\kappa'=(\kappa'_1\ldots,\kappa'_{n-1})$ the roots of $\det(h_{\alpha\beta}-t g_{\alpha\beta})
=0\, (1\leq \alpha,\beta\leq n-1).$ Notice that $\kappa'$ do not denotes the first $n-1$ principal curvatures of $\Sigma.$
For an arbitrary fixed $x\in\partial\Omega,$
let $\tau_1, \ldots, \tau_{n-1}\in T_x\partial\Omega$ be a basis of vectors
that diagonalize $h_{\alpha\beta}$ with
respect to the inner product defined by $g_{\alpha\beta}.$ Then the function curvature $S_r$ of 
the hypersurface $\Sigma$ at $X(x)$ is given by (see, for instance, \cite{ALM})
\begin{align}
\label{Sr-boundary}
\begin{split}
S_r=S_{r-1}(\kappa')\nabla_{nn} u+ D 
\end{split}
\end{align}
where $D$ depends only on $u, \nabla u$ and the tangential and mixed second derivatives of $u.$ Therefore
\begin{align*}
S_{r-1}(\kappa')=\frac{\partial S_r}{\partial \nabla_{nn}u}=\frac{\partial F}{\partial a_{ij}}\frac{\partial a_{ij}}{\partial \nabla_{nn}u}
=g^{nn}\frac{\partial F}{\partial a_{nn}}>0,
\end{align*}
by ellipticity. In particular,  $(\kappa',0)\in \Gamma_{r-1}\subset\mathbb{R}^n.$ 
Now we adapt the techniques used in \cite{CNSIII} and \cite{GUAN-2},
which are based on a brilliant idea introduced by Trudinger in \cite{TRU-2}. First we show that 
an upper bound on $\nabla_{nn} u$ on $\partial \Omega$ amounts to a lower bound on $S_{r-1}(\kappa')$
on $\partial \Omega$ by a uniform positive quantity.

Let $\Gamma'_{r-1}$ be the projection of $\Gamma_{r-1}$ into $\mathbb{R}^{n-1}$ and
denote by $\tilde d(x)$ the distance from $\kappa'(x)$ to $\partial\Gamma'_{r-1}.$
In what follows, we estimate $\nabla_{nn} u$ at a point $x_0$ of $\partial\Omega$ where $\tilde d$ 
is minimum.
So, let $x_0\in\partial\Omega$ be a point where $\tilde d$ attains its minimum.
As above, choose a local frame field $\tau_1, \ldots, \tau_{n-1}$  on $\partial\Omega$ 
around $x_0$ which is orthogonal with respect to the inner product given by
$g_{\alpha\beta}$ and that diagonalizes $h_{\alpha\beta}$ at $x_0.$ Let 
$\tau_1, \ldots, \tau_{n-1}, e_n$ be the frame field on $\mathbb{S}^n$  obtained by parallel translation 
of the local frame field $\tau_1, \ldots, \tau_{n-1}$ along the geodesic
perpendicular to $\partial\Omega$ and  $e_n$ denotes the parallel translation of 
the unit normal field on $\partial\Omega.$ Choose the first $n-1$ indices so that
\begin{equation*}
\kappa'_1\leq \ldots\leq \kappa'_{n-1}.
\end{equation*}
Using Lemma 6.1 of \cite{CNSIII}, we can find a vector $\gamma'=(\gamma_1,\ldots , \gamma_{n-1})
\in\mathbb{R}^{n-1}$ such that
\begin{align*}
\gamma_1\geq \cdots \geq\gamma_{n-1}\geq 0,\quad\sum_{\alpha<n}\gamma_\alpha=1
\end{align*}
and
\begin{equation}
\label{D-X-0}
\tilde d(x_0)=uw\sum_{\alpha<n}\gamma_{\alpha}\kappa'_\alpha(x_0) =
\sum_{\alpha<n}\gamma_{\alpha}\big(u\sigma_{\alpha\alpha}+\nabla_{\alpha\alpha}u\big) (x_0),
\end{equation}
where $\sigma_{\alpha\beta}=\langle \tau_\alpha, \tau_\beta\rangle$ and we have used that 
$g_{\alpha\beta}=\delta_{\alpha\beta}.$  Here we are also using that the distance
of $\kappa'$ and $uw\kappa'$ to $\partial\Gamma'_{r-1}$ is equal. Furthermore,
\begin{equation}
\label{CASA-GAMA-LINHA}
\Gamma'_{r-1}\subset \left\{\lambda'\in\mathbb{R}^{n-1}\, :\, \gamma'\cdot \lambda'>0  \right\}.
\end{equation}
It follows by Lemma 6.2 of \cite{CNSIII}, with $\gamma_n=0,$ that for all $x\in\partial
\Omega$ sufficiently near $x_0$  we have
\begin{align}
\label{DES-T}
\sum_{\alpha<n}\gamma_{\alpha}\big(u\sigma_{\alpha\alpha} +\nabla_{\alpha\alpha} u\big)(x)
\geq uw\sum_{\alpha<n}\gamma_\alpha\kappa'(x)\geq \tilde d(x)\geq\tilde d(x_0),
\end{align}
where we have used \eqref{CASA-GAMA-LINHA} and $|\gamma'|\leq 1$ in the second inequality. 
Then
\begin{align}
\label{DES-B-D}
\begin{split}
\nabla_n u (x)\sum_{\alpha<n}&\gamma_\alpha B_{\alpha\alpha}(x)
= \sum_{\alpha<n}\gamma_\alpha \nabla_{\alpha\alpha}\varphi(x)
-\sum_{\alpha<n}\gamma_\alpha \nabla_{\alpha\alpha}u(x)
\\ &\leq \sum_{\alpha<n}\gamma_\alpha\big(\varphi\sigma_{\alpha\alpha} 
+\nabla_{\alpha\alpha}\varphi\big)(x)-\tilde d(x_0),
\end{split}
\end{align}
where we have used \eqref{DES-T} in the last inequality.

Since the matrix $\{\underline u\sigma_{\alpha\beta} +\nabla^2_{\alpha\beta}\underline u\}$ is positive definite in
a neighbourhood of $\partial\Omega,$ it follows that
$\kappa'[\underline u]:=(\underline u\sigma_{11} +\nabla_{11}\underline u, \ldots, 
\underline u\sigma_{(n-1)(n-1)} +\nabla_{(n-1)(n-1)}\underline u)(x_0)$ belongs to $\Gamma'_{r-1}.$ 
We may assume
\begin{equation*}
\tilde d(x_0)<\frac{1}{2} \textrm{dist}\big(\kappa'[\underline u], \partial\Gamma'_{r-1}\big),
\end{equation*}
otherwise we have a uniform positive lower bound for $S_{r-1}(\kappa')(x_0)$ and
\eqref{NORMAL-NORMAL} follows directly from \eqref{Sr-boundary}. Thus,  we conclude from \eqref{D-X-0} and 
Lemma 6.2 of \cite{CNSIII} that
\begin{align*}
\nabla_n(u-\underline u) (x_0)\sum_{\alpha<n}&\gamma_\alpha B_{\alpha\alpha}(x_0)=
\sum_{\alpha<n}\gamma_\alpha\nabla_{\alpha\alpha}\underline u(x_0)
-\sum_{\alpha<n}\gamma_\alpha \nabla_{\alpha\alpha}u(x_0)
\\ & \geq\tilde d\big(\kappa'[\underline{u}]\big)-\tilde d(x_0) 
> \frac{1}{2} \tilde d\big(\kappa'[\underline{u}]\big)>0.
\end{align*}
As $\nabla_n(u-\underline u)\geq 0$ on
$\partial\Omega,$ then $\sum_{\alpha<n}\gamma_\alpha B_{\alpha\alpha}(x_0)>0$ and 
we conclude that there exist uniform positive
constants $ c,\delta>0,$ such that
\begin{equation*}
\sum_{\alpha<n}\gamma_\alpha B_{\alpha\alpha}(x)\geq c>0, 
\end{equation*}
for every $x\in \Omega$ satisfying $\textrm{dist}_{\mathbb{S}^n}(x, x_0)<\delta.$
Hence we may define the function
\begin{align}
\label{CARA-MU}
\mu(x) =\frac{1}{\sum_{\alpha<n}\gamma_\alpha B_{\alpha\alpha}(x)}
\left(\sum_{\alpha<n}\gamma_\alpha\big(\varphi\sigma_{\alpha\alpha}  
+\nabla_{\alpha\alpha}\varphi\big)(x)-\tilde d(x_0)\right),
\end{align}
for $x\in \Omega_{\delta}=\{ x\in\Omega\, :\, \textrm{dist}_{\mathbb{S}^n}(x,x_0)<\delta\}.$
It follows from \eqref{DES-B-D} that $ \nabla_n u \leq \mu$
on $\partial\Omega\cap\partial\Omega_{\delta}$ for a uniform constant $\delta>0.$
On the other hand, \eqref{D-X-0} implies that $\nabla_n u (x_0) = \mu(x_0).$ 
Then we may proceed as it was done for
the mixed normal-tangential derivatives to get the estimate $\nabla_{nn}u(x_0)
\leq C,$ for a uniform constant $C.$
In fact, redefining the function $\Phi$ given in \eqref{CARA-W} by replacing 
$\nabla_{k}(u-\varphi)$ for $\mu-\nabla_n u,$ i.e., defining
\begin{equation}
\label{CARA-W-2}
\Phi=\mu-\nabla_n u-\frac{K}{2}\sum_{l<n}\big(\nabla_l (u-\varphi)\big)^2,
\end{equation}
we conclude from the uniform bound $|\nabla^2\mu|\leq C$ that inequality \eqref{LW} 
remain valid for this new function $\Phi.$ Defining $\bar\Theta$
as in \eqref{cara-w-barra}, clearly inequality
\eqref{INEL-THETA} remains true. Finally, as $ \nabla_n u \leq \mu$
on $\partial\Omega\cap\partial\Omega_{\delta}$  the function $\bar{\Phi}$ defined in
\eqref{cara-w-barra} satisfies $\bar{\Phi}\geq 0$ on 
$\partial\Omega\cap\partial\Omega_{\delta}.$ Therefore, 
proceeding as above we get the uniform bound
\begin{equation}
\label{U-NUNU-X-0}
\nabla_{nn} u(x_0)\leq C.
\end{equation}

Hence, it follows from the previous estimates that the principal curvatures
 $\kappa_{\Sigma}[X(x_0)]=(\kappa_1, \ldots ,\kappa_n)(x_0)$ of $\Sigma$ at $X(x_0)$ is contained in an
 {\it a priori} bounded subset of $\Gamma_{r}\subset\Gamma_{r-1}.$
Therefore, as
\begin{equation}
H_r^{1/r}(\kappa[u])=\Upsilon\geq \inf_{\Delta_L}\Upsilon>0
\end{equation}
and $H_r=0$ on $\partial\Gamma_r,$ it follows that $\textrm{dist}\big((\kappa_1,\ldots, \kappa_{n-1})(x_0),
\partial\Gamma'_{r-1}\big)\geq \bar{c}_0>0$ for a uniform constant $\bar{c}_0>0.$ 
On the other hand, by  Lemma 1.2 of \cite{CNSIII}, the principal curvatures 
$\kappa_{\Sigma}=(\kappa_1, \ldots ,\kappa_n)$ of $\Sigma$ behave like
\begin{align}
\label{ki}\kappa_\alpha &=\kappa'_\alpha+o(1), \quad  1\leq \alpha\leq n-1,\\
\label{kn}\kappa_n &=\frac{h_{nn}}{g_{nn}}\left(1+O\left(\frac{1}{h_{nn}}\right)\right),
\end{align}
as $|h_{nn}|\rightarrow\infty,$ where $o(1)$ and $O(1/h_{nn})$ are uniform, depending only on 
$\kappa'_1, \ldots, \kappa'_{n-1}$ and the bounds on  $|u|,$ $|\nabla u|$ and 
$|\nabla_{\alpha n} u|,$  $(1\leq \alpha\leq n-1).$
Then there exists a uniform constant $N_0$ such that, if $\nabla_{nn} u(x_0)\geq N_0$ 
the distance of $\big(\kappa_1,\ldots,\kappa_{n-1}\big)(x_0)$ to $\kappa'(x_0)$ 
is less then $\bar{c}_0/2,$ where $\bar{c}_0$ is the constant given above. In particular, if $\nabla_{nn} u(x_0)\geq N_0$
then $\tilde{d}(x_0)\geq c_0$ for a uniform constant $c_0>0,$ which implies that $S_{r-1}(\kappa')$ admits 
itself a uniform positive bound on $\partial\Omega$ and, in this case, \eqref{NORMAL-NORMAL} 
follows from \eqref{Sr-boundary}.
This establish \eqref{NORMAL-NORMAL} and completes the proof of
Theorem \ref{teorema1-2}.

In \cite{CNSIV} it is also shown how to derive the global estimates for $|\nabla^2 u|$ on $\bar\Omega$
from its bound on the boundary $\partial\Omega,$ if $\Upsilon$ satisfy \eqref{derivada-ro}. 
Then we have the following result.
\begin{theorem}
\label{teorema3}
Let $u\geq \underline{u}$ be an admissible solution of \eqref{equation4} and suppose 
that $\Upsilon$ satisfy \eqref{derivada-ro}. Then we have the estimate
\begin{equation}
\label{S3-2}
\| u\|_{C^{2}(\bar{\Omega})}\leq C \quad \textrm{in } \bar\Omega,
\end{equation}
where $C$ depends on $ \inf_{\bar\Omega}\underline u, \|\underline{u}\|_{C^2(\bar\Omega)},
\|\psi^t\|_{C^2(\Delta_L)},$ the convexity of $\bar\Sigma$ in a neighbourhood of $\partial\Omega$
and other known data.
\end{theorem}


\section{Proof of Theorem \ref{teorema1}}
\label{section4}

In this section we complete the proof of Theorem \ref{teorema1} by applying 
the method of continuity and a degree theory argument
with the aid of the {\it a priori} estimates we have established. Our proof
is inspired in \cite{CNSI} and \cite{SU}, where Monge-Ampère type equations are treated.
 
Here we will not deal with equation \eqref{equation4} because 
$G_u$ is positive and can not be bounded easily. Then we need to express \eqref{equation4}
in a different form. Setting  $v=-\ln \rho=\ln u,$
the matrix  $A[u]=[a_{ij}]$ can be written in terms of $v$ by
\begin{equation}
\label{g-v}
a_{ij}=\frac{e^v}{w}\big(\delta_{ij}+\gamma^{ik}\nabla_{kl}v\gamma^{jl}\big)
\end{equation}
where
\begin{equation}
\label{h-v}
w=\sqrt{1+|\nabla v|^2}, \quad \gamma^{ij}=\delta_{ij}-\frac{\nabla_i v\nabla_j v}{w(1+w)}.
\end{equation}
Hence, for $\Upsilon=\Psi^t,$  \eqref{equation4} takes the form
\begin{align}
\begin{split}
\label{equation5}
H(\nabla^2 v,\nabla v, v)&=\Psi^t(X(x))=e^{3(v-\underline v)} \big(t\psi(x)+(1-t)
\underline{\psi}(x)\big)\quad\textrm{in}\, \Omega,\\
v&=-\ln \phi \quad\textrm{on}\, \partial\Omega,
\end{split}
\end{align}
Notice that $\underline v=-\ln\bar\rho=\ln \underline{u}$
is a strictly subsolution of \eqref{equation5} fot $t>0$ and it is a solution for $t=0.$
Moreover, as 
\begin{align*}
\begin{split}
H_v=F^{ij} a_{ij}=\sum f_i\kappa_i \leq f(\kappa) =\Psi^t
\end{split}
\end{align*}
and 
\begin{align*}
\Psi^t_v= \frac{\partial \Psi^t }{\partial v}=3\Psi^t,
\end{align*}
it follows that $H_v-\Psi^t_v\leq 0.$ Then we can apply the comparison principle to equation \eqref{equation5} 
to conclude that any solution $v^t$ of \eqref{equation5} for $t>0$ satisfy $v^t> \underline{v}.$ 
Hence  Theorem \ref{teorema3} can be applied and we get the $C^2$ estimates for any solution $v^t$ of \eqref{equation5}.
Therefore the holder estimates follows from the Evans-Krylov Theorem and 
we can apply the continuity method to conclude that there exist a unique solution $v^0$ of
\eqref{equation5} for $t=1.$ Now we consider the family of equations ($s\in [0,1]$)
\begin{align}
\begin{split}
\label{equation6}
H(\nabla^2 v,\nabla v, v)&=\Xi^s(X(x))=s\psi(x)+(1-s)e^{3(v-\underline v)}
\underline{\psi}(x)\quad\textrm{in}\, \Omega,\\
v&=-\ln\phi \quad\textrm{on}\, \partial\Omega.
\end{split}
\end{align}
From Theorem \ref{teorema3}, the Evans-Krylov Theorem and by the standard regularity theory for
 second order uniformly elliptic equations we can get higher order estimate 
\begin{align}
\label{c4forvt}
\| v^s\|_{C^{4,\alpha} (\bar\Omega)}< \bar C\quad\textrm {independent of } s,
\end{align}
for any solution $v^s$ of equation \eqref{equation6} satisfying $v^s\geq \underline v.$ We also point out that, if $s>0$ and
$v^s\geq \underline{v}$ is a solution of \eqref{equation6} then $v^s$ is a supersolution of \eqref{equation5} for
$t=s.$ In particular, we have the strictly inequality $v^s>\underline v$ in this case.

Let $C_0^{4,\alpha}(\bar\Omega)$ be the subspace of $C^{4,\alpha}(\bar\Omega)$ consisting of functions vanishing on
the boundary. Consider the cone
\begin{align*}
\mathcal{O}  = \{ & z\in C_0^{4,\alpha}(\bar\Omega)  :  z>0 \textrm{ in } \Omega, \, 
 \nabla_n z >0   \textrm{ on }  \partial\Omega, \\ & z+\underline{v} \textrm{ is admissible} \textrm{ and }
 \|z\|_{C^{4,\alpha}(\bar\Omega)}\leq \bar C+\|\underline{v}\|_{C^{4,\alpha}(\bar\Omega)}\}, 
\end{align*}
where $\bar C$ is the constant given in \eqref{c4forvt}. Now we construct a map from $\mathcal{O}\times [0,1]$ to
$C^{2,\alpha}(\bar\Omega)$ given by
\begin{equation*}
M_s[w]=H(\nabla^2(z+\underline{v}), \nabla (z+\underline{v}), z+\underline{v})-
\Xi^s(z+\underline{v}), \quad z\in\mathcal{O},
\end{equation*}
where $\Xi^s$ is the function given in \eqref{phi-t}:
\begin{equation*}
\Xi^s(z+\underline{v})= t\psi(x)+(1-t)e^{3z}\psi(x).
\end{equation*} 
Clearly, $z$ is a solution of $M_s[z]=0$ iff $v^s=z+\underline{v}$ is a solution of \eqref{equation6}.
In particular,  $z^0=v^0-\underline{v}$ is the unique solution of $M_0[z]=0$ and $z^0\in\mathcal{O}.$
Moreover, there is no solution of $M_s[z]=0$ on $\mathcal{O}$ for any $s.$ Therefore, the degree of $M_s$ on 
$\mathcal{O}$ at $0$ $\deg(M_s,\mathcal{O},0)$ is well defined and independent of $s.$ For more details, we refer the
reader to \cite{LI1} and \cite{LI2}.

Now we compute  $\deg(M_0,\mathcal{O},0).$ We know that  $M_0[z]=0$ has a unique solution $z^0$ in $\mathcal{O}.$
The Fréchet derivative of $M_0$ at $z^0$ is a linear elliptic operator from $C^{4,\alpha}_0(\bar\Omega)$ to
$C^{2,\alpha}(\bar\Omega)$,
\[
M_{0,z^0}(h)=H^{ij}|_{v^0}\nabla_{ij}h+H^i|_{v^0}\nabla_i h+(H_{v}|_{v^0}-\Xi^0_v|{v^0})h.
\]
By \eqref{derivada-ro}, $H_{v}|_{v^0}-\Xi^0_v|{v^0}<0.$ So $M_{0,z^0}$ is invertible. By the theory in \cite{LI1}, we
can see
\[
\deg(M_0,\mathcal{O},0)=\deg(M_{0, z^0},B_1,0)=\pm 1\neq 0,
\]
where $B_1$is the unit ball of $C_0^{4,\alpha}(\bar\Omega).$ Therefore
\[
\deg(M_s,\mathcal{O},0)\neq 0 \quad \textrm{for all}\,  s\in [0,1].
\]
Then equation $M_s[z]=0$ has at least one solution for any $s\in[0,1].$ In particular,
the function $v^1=z^1-\underline{v}$ is then a solution of \eqref{equation5}. Therefore
$\rho = e^{-v^1}$ is a solution of \eqref{equation}.


\section{Appendix}

For completeness, we present here the prove that inequality \eqref{delta} implies
inequality \eqref{delta2}, i.e., the existence of some $1\leq \gamma\leq n$ such that
\begin{equation}
\label{S5-1}
\sum_{l<n}\eta_{l\gamma}^2< K^{-2}.
\end{equation}
implies that 
\begin{equation}
\label{S5-2}
\sum_{l<n}^{n-1}\eta_{l\alpha}^2\geq\epsilon_0
\end{equation}
for all $\alpha\neq \gamma$ and for a uniform positive constant $\epsilon_0>0$ that does
not depend on $K.$ For simplicity, let us assume $\gamma=1.$ 
As $P=[\eta_{ij}]$ is an orthogonal matrix, $\eta_{ij}$ is the cofactor of index $(j,i)$ of $P.$
Developing the determinant of $P$ with respect to the first line we get
\begin{align}
\label{lema-c1}
\begin{split}
1&= |\det(\eta_{ij})| 
\leq  |\eta_{1n}\eta_{n1}|+\sum_{l<n}|\eta_{l1}\eta_{1l}|\\ &\leq  |\eta_{1n}\|\eta_{n1}|+
\Big(\sum_{l<n}(\eta_{l1})^2\Big)^{1/2}\Big(\sum_{l<n}(\eta_{1l})^2\Big)^{1/2} 
\\ &\leq  |\eta_{n1}|+K^{-1}.
\end{split}
\end{align}
Now we develop the cofactor $\eta_{n1}$ with respect to the $\alpha^{th}$ line, $2\leq\alpha\leq n,$ 
to obtain
\begin{align*}
|\eta_{n1}|& = \Big| \sum_{l<n} \eta_{l\alpha} \zeta_{\alpha l}\Big|
\leq (n-2)!\sum_{l<n} | \eta_s^\alpha| (n-2)! 
\\ & \leq (n-2)!(n-1)^{1/2}\Big(\sum_{s=1}^{n-1} (\eta_s^\alpha)^2\Big)^{1/2}.
\end{align*}
where $\zeta_{l\alpha}$ denotes the cofactor of index $(\alpha, l)$ of the $[n-1]$ matrix 
$(\eta_{ij}),$ where $2\leq i\leq n$ and $1\leq j\leq n-1.$ Thus, replacing this inequality into
\eqref{lema-c1} we obtain
\begin{align*}
1\leq (n-2)!(n-1)^{1/2}\Big(\sum_{l<n} (\eta_{l\alpha})^2\Big)^{1/2}+K^{-1}.
\end{align*}
Therefore, for $K\geq 2,$
\begin{align*}
\Big(\sum_{l<n} (\eta_{l\alpha})^2\Big)^{1/2}
\geq \frac{1}{2(n-1)^{1/2}(n-2)!},
\end{align*}
which proves \eqref{S5-2}. 


\subsection*{Acknowledgements} This work has been partially supported by CAPES.
Part of this work was written while the author was  
visiting the Universidad de Murcia at Murcia, Espain.
He thanks that institution for its hospitality.


\end{document}